\documentclass[reqno,10pt]{amsart}
\usepackage{epsfig}
\usepackage{color}
\usepackage{enumerate} 
\usepackage{times}
\usepackage{amsfonts,amsmath,amssymb}
\usepackage{mathrsfs}
\usepackage[latin1]{inputenc}
\usepackage[T1]{fontenc}
\numberwithin{equation}{section}
\newtheorem{theorem}{Theorem}[section]
\newtheorem{lemma}[theorem]{Lemma}
\newtheorem{remark}[theorem]{Remark}

\newtheorem{Corollary}[theorem]{Corollary}
\allowdisplaybreaks

\newtheorem{Prop}{Proposition}[section]

\def\qed{\hfill$\square$\par \bigskip}

\newcommand{\R}{\mathbb{R}}

\newcommand{\norm}[1]{\|#1\|}
\newcommand{\abs}[1]{\left\vert#1\right\vert}

\newcommand{\para}[1]{\left(#1\right)}

\def\beq{\begin{equation}}
\def\eeq{\end{equation}}
\renewcommand{\leq}{\leqslant}
\renewcommand{\geq}{\geqslant}
\newcommand{\bea}{\begin{eqnarray}}
\newcommand{\eea}{\end{eqnarray}}
\newcommand{\beas}{\begin{eqnarray*}}
\newcommand{\eeas}{\end{eqnarray*}}

{
\newcommand{\bel}{\begin{equation} \label}
\newcommand{\ee}{\end{equation}}

\newcommand{\pd}{\partial}

\newcommand{\cB}{\mathcal{B}}

\newcommand{\cC}{\mathcal{C}}

\newcommand{\cH}{\mathcal{H}}

\newcommand{\supp}{\mathrm{supp}}

\usepackage{graphicx}

\newcommand{\re}{\textrm{Re}}
\newcommand{\im}{\textrm{Im}}

\usepackage{times}
\begin{document}
\title[H\"older stably determining the time-dependent electromagnetic potential]{H\"older stably determining the time-dependent electromagnetic potential of the Schr\"odinger equation}
\author{
Yavar Kian$^*$ and
Eric Soccorsi$^*$
}
\begin{abstract}
We consider the inverse problem of determining the time and space dependent electromagnetic potential of the Schr\"odinger equation in a  bounded domain of $\R^n$, $n\geq 2$, by boundary observation of the solution over the entire time span. Assuming that the divergence of the magnetic potential is fixed, we prove that the electric potential and the magnetic potential can be H\"older stably retrieved from these data, whereas stability estimates for inverse time-dependent coefficients problems of evolution partial differential equations are usually of logarithmic type.


\end{abstract}


\maketitle

\renewcommand{\thefootnote}{\fnsymbol{footnote}}
\footnotetext{\hspace*{-5mm} 
\begin{tabular}{@{}r@{}p{13cm}@{}} 
$^*$& Aix Marseille Universit\'e, Universit\'e de Toulon, CNRS, CPT, 
Marseille, France.
\end{tabular}}
\section{Introduction}
\setcounter{equation}{0}

\subsection{Statement of the problem}
Let $\Omega$ be a bounded, simply connected domain of $\R^n$, $n \geq 2$, with $C^2$ boundary $\pd \Omega$. 
For $T \in (0,+\infty)$, we consider the initial boundary value problem (IBVP)
\bel{1.1}
\left\{
\begin{array}{ll}
\para{i\pd_t+\Delta_{A}+q} u=0  & \mbox{in}\ Q:=(0,T) \times \Omega,\\
u(0,\cdot)=0 & \mbox{in}\ \Omega,\\
u=g & \mbox{on}\ \Sigma:=(0,T) \times \Gamma,
\end{array}
\right.
\ee
where $\Delta_{A}$ is the Laplace operator  $(\nabla + i A(t,x) ) \cdot (\nabla  + i A(t,x) )$, associated with the real-valued magnetic potential $A :=(a_j)_{1 \leq j \leq n} \in W^{1,\infty}(Q;\R)^n$, i.e.
\bel{DelA}
\Delta_{A}:=\sum_{j=1}^n
\para{\pd_{x_j}+i a_j(t,x)}^2 = \Delta+2i A(t,x) \cdot \nabla +i (\nabla \cdot A(t,x)) -|A(t,x)|^2
\ee
and $q \in L^{\infty}(Q;\R)$ is a real-valued electric potential. Here and in the remaining part of this text, we denote by $\nabla:=(\pd_{x_1},\ldots,\pd_{x_n})^T$ the gradient operator with respect to the spatial variable $x:=(x_1,\ldots,x_n) \in \R^n$, the symbol $\cdot$ (resp., $| \cdot |$) stands for the Euclidian scalar product (resp., norm) in $\R^n$, and the divergence operator with respect to $x \in \R^n$ is represented by the notation $\nabla \cdot$.

For all $s,r\in(0,+\infty)$  and for $X$ being either $\Omega$ or $\pd \Omega$, we equip the functional spaces
$H^{r,s}((0,T)\times X) := H^r(0,T;L^2(X)) \cap L^2(0,T;H^s(X))$ with the following norm
$$\norm{u}_{H^{r,s}((0,T) \times X)}^2 :=\norm{u}_{H^r(0,T;L^2(X))}^2+\norm{u}_{L^2(0,T;H^s(X))}^2,$$
and we write $H^{r,s}(Q)$ (resp., $H^{r,s}(\Sigma)$) instead of $H^{r,s}((0,T)\times \Omega)$ (resp., $H^{r,s}((0,T)\times \pd \Omega)$). Then, for all
$$g \in \cH(\Sigma) := \left\{g\in H^{\frac{9}{4}, \frac{3}{2}}(\Sigma);\ g(0,\cdot)=\pd_t g(0,\cdot)=0\ \mbox{on}\ \pd \Omega \right\}, $$
we establish in Proposition \ref{p1} below, that there exists a unique solution $u_g \in  H^{1,2}(Q)$ to \eqref{1.1} and that the mapping $g \mapsto u_g$ is continuous. As a corollary the Dirichlet-to-Neumann (DN) operator associated with \eqref{1.1}, defined by
\bel{DN}
\begin{array}{cccc}
\Lambda_{A,q}: & \cH(\Sigma) & \to & L^2(\Sigma) \\ & g & \mapsto & (\pd_\nu+ i A \cdot \nu ) u_g,
\end{array}
\ee
where $\nu(x)$ denotes the unit outward normal to $\pd \Omega$ at $x$, is bounded. The main purpose of this paper is to examine the stability issue in the inverse problem of determining the electromagnetic potential $(A,q)$ from the knowledge of $\Lambda_{A,q}$.

However, there is a natural obstruction to uniqueness in this problem, in the sense that the mapping $(A,q) \mapsto \Lambda_{A,q}$ is not injective.
This can be seen from the identity
$$ i \pd_t + \Delta_A + q= e^{-i \phi} ( i \pd_t + \Delta_{\tilde{A}} + \tilde{q}  ) e^{i \phi},\  \phi \in W^{3,\infty}(Q) $$
arising from \eqref{DelA} with $(\tilde{A},\tilde{q})=G_{A,q}(\phi):=(A - \nabla \phi, q+\pd_t \phi)$, which entails that 
$(i \pd_t + \Delta_{\tilde{A}} + \tilde{q}) e^{i \phi} u_g = e^{i \phi} (i \pd_t + \Delta_A + q) u_g = 0$ for all $g \in \cH(\Sigma)$.
Therefore, if $\phi$ vanishes on $\Sigma$ then it is apparent that $e^{i \phi} u_g$ is the $H^{1,2}(Q)$-solution to \eqref{1.1}, where $(\tilde{A},\tilde{q})$ is substituted for $(A,q)$. Consequently, it holds true that
$$ \Lambda_{\tilde{A},\tilde{q}} g = (\pd_\nu+ i \tilde{A} \cdot \nu ) e^{i \phi} u_g = e^{i \phi} (\pd_\nu+ i A \cdot \nu ) u_g =
\Lambda_{A,q} g, $$
despite of the fact that $(\tilde{A},\tilde{q})$ does not coincide with $(A,q)$ whenever $\phi$ is not uniformly zero in $Q$. Otherwise stated, since the DN map \eqref{DN} is invariant under the gauge transformation $(A,q) \mapsto G_{A,q}(\phi)$ associated with
$\phi \in W_*^{3,\infty}(Q) := \{ \phi \in W^{3,\infty}(Q);\ \phi_{\vert \Sigma}=0 \}$, then it is hopeless to retrieve $(A,q)$ through $\Lambda_{A,q}$ and the best we can expect is to determine the {\it gauge class} 
$G_{A,q}(W_*^3(Q)):=\{ G_{A,q}(\phi),\ \phi \in W_*^{3,\infty}(Q) \}$ of $(A,q)$. Moreover, for any two gauge equivalent electromagnetic potentials $(A,q)$ and $(\tilde{A},\tilde{q})$, there exists a unique $\phi \in W_*^{3,\infty}(Q)$ such that we have $(\tilde{A},\tilde{q})=G_{A,q}(\phi)$ and we notice for each for $t \in (0,T)$ that the function $\phi(t,\cdot)$ is solution to the following elliptic system:
$$ \left\{ \begin{array}{ll} -\Delta \phi(t,\cdot) = \nabla \cdot ( \tilde{A} - A )(t,\cdot) & \mbox{in}\ \Omega, \\
\phi(t,\cdot) =0 & \mbox{on}\ \pd \Omega. \end{array} \right. $$
Therefore, if the time-dependent electromagnetic potential $(A,q)$ can be determined modulo gauge invariance by $\Lambda_{A,q}$ then it is actually possible to recover $(A,q)$ itself provided the divergence $\nabla \cdot A$ is known.

\subsection{What is known so far}

Since inverse problems are of great interest in applied sciences, it is no surprise that the determination of coefficients in partial differential equations such as the magnetic Schr\"odinger equation under study in this article has attracted the attention of numerous mathematicians over the previous decades.

For instance, using the Bukhgeim-Klibanov method \cite{Bukhgeim-Klibanov}, Baudouin and Puel \cite{Baudouin-Puel} proved Lipschitz stable identification of the time independent electric potential in the dynamical (i.e., non stationary) Schr\"odinger equation from a single boundary observation of the solution. Here the measurement can be performed on any subpart of the boundary fulfilling the geometric control property expressed by Bardos, Lebeau and Rauch in \cite{BLR}. This condition was removed by Bellassoued and Choulli in \cite{BC1}, provided the electric potential is {\it a priori} known in a neighborhood of the boundary. We refer to \cite{CS} for the Lipschitz stable reconstruction of the magnetic potential in the Coulomb gauge class by a finite number of boundary measurements of the solution to the Schr\"odinger equation. More recently, in \cite{BAM}, Ben A\"icha and Mejri claimed simultaneous Lipschitz stable determination of the electric potential and the divergence free magnetic potential, from the same type of boundary data.

All the above mentioned results involve a finite number of boundary observations of the solution, performed over the entire time span. This is no longer the case in \cite{BC2} where the magnetic field was stably recovered from the knowledge of the DN map associated with the dynamic  Schr\"odinger equation. In the same spirit, Bellassoued and Dos Santos Ferreira proved stable identification of the electric potential by the DN map associated with the Schr\"odinger equation on a Riemannian manifold in \cite{BD}.
This result was extended in \cite{Be} to simultaneous determination of the electric potential and the magnetic field. We also refer to \cite{DKSU,ER,Ki3} for an extensive treatment of similar inverse problems. We stress out that all the above results were established in a bounded domain and that the analysis carried out in \cite{Baudouin-Puel} (resp. \cite{BC1} and \cite{BC2}) was adapted to the case of unbounded cylindrical domains in \cite{BKS} (resp., \cite{BKS2} and \cite{KPS1,KPS2}). 

All the above mentioned works are concerned with space-only dependent (i.e. time independent) coefficients. Actually, there is only a very small number of papers available in the mathematical literature, dealing with the inverse problem of determining time-dependent coefficients of the Schr\"odinger equation. For instance, it was proved in \cite{E1} that the DN map uniquely determines
the time-dependent electromagnetic potential modulo gauge invariance. The stability issue for the same problem was examined in \cite{CKS}, where the time-dependent electric potential was logarithmic stably recovered from boundary observation for all times and internal measurement at final time, of the solution. More recently, in \cite{Ben2}, this approach was adapted to the case of an electromagnetic potential with sufficiently small time independent magnetic part. To the best of our best knowledge, these two last articles are the only mathematical papers studying the stability issue in the inverse problem of determining time-dependent coefficients of the Schr\"odinger equation. Nevertheless, we point out that similar problems were addressed in \cite{BB,Ben1,CK1,CK2,E2,E3,FK,GK,I1,Ki1,Ki2,Ki4,RS,RR,St1} for either parabolic, hyperbolic, or even time-fractional diffusion equations.

\subsection{Main result}
 

The main result of this paper is the following H\"older stability estimate of the electromagnetic potential entering the Schr\"odinger equation in \eqref{1.1}, with respect to the DN map. 

\begin{theorem}
\label{t1}
Fix $M \in (0,+\infty)$ and for $j=1,2$, let $A_j\in W^{5,\infty}(Q)^n\cap H^6(Q)^n$ and $q_j\in W^{4,\infty}(Q)$ satisfy the three following conditions:
\bel{cond1} 
\pd_x^\alpha A_1(t,x)=\pd_x^\alpha A_2(t,x),\ (t,x)\in\Sigma,\ \alpha\in\mathbb N^n,\ |\alpha|\leq5,
\ee
\bel{cond2}
\nabla \cdot A_1(t,x)=\nabla \cdot A_2(t,x),\ (t,x)\in Q
\ee
and
\bel{cond3}
\sum_{j=1}^2\left( \norm{A_j}_{W^{5,\infty}(Q)^n}+\norm{A_j}_{H^6(Q)^n}+\norm{q_j}_{W^{4,\infty}(Q)} \right) \leq M.
\ee
Then, there exist three positive constants, $r$ and $s$, depending only on $n$, and $C$, depending only on $T$, $\Omega$ and $M$, such that we have
\bel{t1a}
\norm{A_1-A_2}_{L^2(0,T;H^5(\Omega))}\leq C \norm{\Lambda_{A_1,q_1}-\Lambda_{A_2,q_2}}^{r}
\ee
and
\bel{t1b}
\norm{q_1-q_2}_{H^{-1}(Q)} \leq C\norm{\Lambda_{A_1,q_1}-\Lambda_{A_2,q_2}}^{s}.
\ee
\end{theorem}
In Theorem \ref{t1} and the remaining part of this article, the DN maps $\Lambda_{A_j,q_j}$, $j=1,2$, lie in the space $\cB(\cH(\Sigma),L^2(\Sigma))$ of linear bounded operators
from $\cH(\Sigma)$ into $L^2(\Sigma)$ and $\norm{\cdot}$ denotes the usual norm in $\cB(\cH(\Sigma),L^2(\Sigma))$.

\subsection{Brief comments and outline}

To the author's best knowledge, Theorem \ref{t1} is the only result available in the mathematical literature claiming H\" older stable determination of space and time varying coefficients appearing in an evolution PDE (all the other existing stability estimates derived in this framework are at best of logarithmic type). Moreover, even if the identification of unknown coefficients depending on both time and space variable is of great interest in its own, it is worth mentioning that it can also be linked with the inverse problem of determining a nonlinear perturbation of a PDE. As a matter of fact it was proved in 
 \cite{CK2,I2} by mean of a linearization process that the semilinear term entering a nonlinear parabolic equation can be identified by solving the inverse problem of determining the time-dependent coefficient of a related linear parabolic equation. From this viewpoint there is no doubt that Theorem \ref{t1} is a useful tool for adapting this strategy to the case of semilinear Schr\"odinger equations.
 
The remaining part of this article is organized as follows. In the coming section, Section \ref{sec-dir}, we study the well-posdeness of problem \eqref{1.1} and we prove that the DN map \eqref{DN} is well defined as a linear bounded operator from $\cH(\Sigma)$ into
 $L^2(\Sigma)$. In Section \ref{sec-GO} we build a set of geometrical optics solutions to \eqref{1.1} which are the main tool for deriving Theorem \ref{t1}. Finally, the proof of the stability estimate \eqref{t1a} is presented in Section \ref{sec-sema}, whereas the one of \eqref{t1b} is given in Section \ref{sec-seel}.
 

\section{Analysis of the forward problem}
\label{sec-dir}
\setcounter{equation}{0}
The main result of this section is the following existence and uniqueness result for the IBVP \eqref{1.1}.

\begin{Prop}
\label{p1} 
For $M \in (0,+\infty)$, let $A\in W^{2,\infty}(Q;\R)^n$ and $q\in W^{1,\infty}(Q;\R)$ satisfy 
\bel{p1a}
\norm{A}_{W^{2,\infty}(Q)^n}+\norm{q}_{W^{1,\infty}(Q)}\leq M.
\ee
Then, for every $g\in \cH(\Sigma)$, the system \eqref{1.1} admits a unique solution $u\in  H^{1,2}(Q)$. Moreover, there exists a positive constant $C$, depending only on $M$, $T$ and $\Omega$, such that we have
\bel{p1b} 
\| u \|_{H^{1,2}(Q)} \leq C \| g\|_{\cH(\Sigma)}.
\ee
\end{Prop}

With reference to \eqref{DN} and the continuity of the trace operator from $H^{1,2}(Q)$ into $L^2(\Sigma)$, Proposition \ref{p1} immediately entails the following:

\begin{Corollary}
\label{cor}
Under the conditions of Proposition \ref{p1}, the DN map $\Lambda_{A,q}$ is well defined by \eqref{DN} and acts as a bounded operator from $\cH(\Sigma)$ into $L^2(\Sigma)$.
\end{Corollary}

The proof of Proposition \ref{p1} can be found in Section \ref{sec-p1} by mean of an existence and uniqueness result for the IBVP \eqref{1.1} with homogeneous Dirichlet boundary condition and suitable source term, stated in Section \ref{sec-eur}. As a preamble, we recall that the sesquilinear form associated with the operator $-\Delta_{A(t,\cdot)}+ q(t,\cdot)=-(\nabla + i A(t,\cdot) ) \cdot (\nabla + i A(t,\cdot)) +q(t,\cdot)$ is $H^1(\Omega)$-elliptic with respect to $L^2(\Omega)$, uniformy in $t \in (0,T)$.

\subsection{$H^1(\Omega)$-ellipticity with respect to $L^2(\Omega)$.}

We define the magnetic gradient $\nabla_A$, associated with $A \in L^{\infty}(Q)^n$, by
\bel{ell0a}
\nabla_A u(t,x) := (\nabla + i A(t,x) ) u(x),\ u \in H_0^1(\Omega),\ (t,x) \in Q,
\ee
and for $q \in L^{\infty}(Q)$, we introduce the sesquilinear form
\bel{ell0b}
a(t;u,v):=\int_\Omega \nabla_A u(t,x) \cdot \overline{\nabla_A v(t,x)}dx+\int_\Omega q(t,x) u(x) \overline{v(x)} dx,\ u,v \in H_0^1(\Omega).
\ee
Then, the H\"older inequality yields
$$
\| \nabla_A u(t,\cdot) \|_{L^2(\Omega)^n}^2   \geq  \frac{\| \nabla u \|_{L^2(\Omega)^n}^2}{2}  -2 \| A \|_{L^\infty(Q)^n}^2 \| u \|_{L^2(\Omega)}^2, 
$$ 
for every $u \in H_0^1(\Omega)$ and $t \in (0,T)$, so we get
\bel{ell1} 
a(t;u,u)+ \lambda  \| u \|_{L^2(\Omega)}^2 \geq \alpha \| u \|_{H^1(\Omega)}^2,\ u \in H_0^1(\Omega),\ t \in (0,T),
\ee
with $\lambda := \frac{1}{2} + \| q\|_{L^\infty(Q)}^2 + 2 \| A \|_{L^\infty(Q)^n}^2$ and $\alpha := \frac{1}{2}$.

\subsection{Existence and uniqueness result}
\label{sec-eur}
The proof of Proposition \ref{p1} essentially boils down to the following existence and uniqueness result for the following IBVP associated with a suitable source term $F$:
\bel{eq2}
\left\{\begin{array}{ll}
(i\pd_t +\Delta_{A} + q ) v = F & \mbox{in}\ Q,\\ 
 v(0,\cdot)=0 &\mbox{in}\ \Omega,\\ 
 v=0 &\mbox{on}\ \Sigma.
\end{array}
\right.
\ee

\begin{lemma}
\label{l1} 
Let $M$, $A$ and $q$ be the same as in Proposition \ref{p1} and
let $F \in H^1(0,T;L^2(\Omega))$ verify $F(0,\cdot)=0$ a.e. in $\Omega$. Then, the system
\eqref{eq2} admits a unique solution $v \in \cC([0,T],H_0^1(\Omega)\cap H^2(\Omega) )\cap \cC^1([0,T],L^2(\Omega))$ satisfying
\bel{l1a}
\| v \|_{\cC([0,T],H^2(\Omega))} +\| v \|_{\cC^1([0,T],L^2(\Omega))}\leq C \| F \|_{H^1(0,T;L^2(\Omega))},
\ee
for some positive constant $C$ depending only on $T$, $\Omega$ and $M$.
\end{lemma}

\begin{proof}
We proceed as in the derivation of \cite[Section 3, Theorem 10.1]{LM1} by applying the Faedo-Galerkin method. Namely, we pick a Hilbert basis $\{ e_k,\ k \in \mathbb N^* \}$  of $H^1_0(\Omega)$ and consider an approximated solution of size $m \in \mathbb N^*:=\{1,2,\ldots \}$ of \eqref{eq2}, of the form
\bel{e0}
v_m(t,x) :=\sum_{k=1}^m g_{k,m}(t) e_k(x),\ (t,x) \in Q,
\ee
where the functions $g_{k,m}$ are defined in such a way that we have
\bel{od1}
\left\{\begin{array}{ll}
i \langle \pd_t v_m(t,\cdot),e_k \rangle_{L^2(\Omega)} -a(t;v_m(t,\cdot),e_k) = \langle F(t,\cdot) ,e_k \rangle_{L^2(\Omega)}, & t \in (0,T), \\
v_m(0,\cdot) = 0, &
\end{array}
\right. 
\ee
for all $k=1,\ldots,m$. 
Since $F \in W^{1,1}(0,T;L^2(\Omega))$ then  \eqref{od1} admits a unique solution $v_m\in W^{1,\infty}(0,T;H^1_0(\Omega))$ such that the function $w_m:=\pd_t v_m$ solves the following Cauchy problem for every $k=1,\ldots,m$:
\bel{od2}
\left\{\begin{array}{ll}
i \langle \pd_t w_m(t,\cdot),e_k \rangle_{L^2(\Omega)} -a(t;w_m(t,\cdot),e_k) = a'(t;v_m(t,\cdot),e_k) + \langle \pd_t F(t,\cdot) , e_k \rangle_{L^2(\Omega)}, & t \in (0,T), \\
w_m(0) = 0. &
\end{array}
\right. 
\ee
Here, for all $t \in (0,T)$ and all $u, v \in H_0^1(\Omega)$, we have set with reference to \eqref{ell0a}-\eqref{ell0b}
\bel{e1}
a'(t;u,v):=i \int_\Omega  \pd_t A(t,x) \cdot  \left( u(x) \overline{\nabla_A v(t,x)} -  \nabla_A u(t,x) \overline{v(x)} \right) dx + 
\int_\Omega \pd_t q(t,x) u(x) \overline{v(x)} dx.
\ee

\noindent 1) The first part of the proof is to establish three {\it a priori} estimates for the functions $v_m$ and $w_m$.

a) To this end, we fix $t \in (0,T)$ and we multiply for each $k \in \{1,\ldots,m\}$ the first line of \eqref{od1} by $\overline{g_{k,m}(t)}$, sum up over $k=1,\ldots,m$, and infer from \eqref{e0} that
$$ i \langle \pd_t v_m(t,\cdot) , v_m(t,\cdot) \rangle_{L^2(\Omega)} - a(t;v_m(t,\cdot),v_m(t,\cdot))=\langle F(t,\cdot),v_m(t,\cdot) \rangle_{L^2(\Omega)}.
$$
Taking the imaginary part of both sides of the above identity then yields
$$
 \frac{d}{d s} \| v_m(s,\cdot) \|_{L^2(\Omega)}^2  = 2 \im \langle F(s,\cdot), \pd_t v_m(s,\cdot) \rangle_{L^2(\Omega)},\ s \in (0,T).
$$
Since $v_m(0,\cdot)=0$, we get upon integrating the above identity over $(0,t)$ that
$$ \| v_m(t,\cdot) \|_{L^2(\Omega)}^2 = 2 \im \int_0^t \langle F(s,\cdot),v_m(s,\cdot) \rangle_{L^2(\Omega)} d s \leq
\int_0^t \| F(s,\cdot) \|_{L^2(\Omega)}^2 ds + \int_0^t \| v_m(s,\cdot) \|_{L^2(\Omega)}^2 d s. $$
Therefore, by Gronwall's lemma, there exists a positive constant $c_0$, depending only on $T$, $\Omega$ and $M$ such that we have:
\bel{e1b} 
\| v_m \|_{L^\infty(0,T;L^2(\Omega))} \leq c_0 \| F \|_{L^2(Q)}.
\ee

b) Similarly, by multiplying the first line of \eqref{od1} by $\overline{g_{k,m}'(t)}$, summing up over $k=1,\ldots,m$, and applying \eqref{e0} once more, we get that
$$ i \| \pd_t v_m(t,\cdot) \|_{L^2(\Omega)}^2 - a(t;v_m(t,\cdot), \pd_t v_m(t,\cdot))=\langle F(t,\cdot), \pd_t v_m(t,\cdot) \rangle_{L^2(\Omega)},\ t \in (0,T).
$$
Upon taking this time the real part in the above identity, we find that
$$
a(s;v_m(s,\cdot),\pd_t v_m(s,\cdot)) + a(s;\pd_t v_m(s,\cdot),v_m(s,\cdot)) = - 2 \re \langle F(s,\cdot),\pd_t v_m(s,\cdot) \rangle_{L^2(\Omega)},\ s \in (0,T),
$$
which may be equivalently rewritten as
$$\frac{d}{d s} a(s;v_m(s,\cdot),v_m(s,\cdot)) = a'(s;v_m(s,\cdot),v_m(s,\cdot)) - 2 \re \langle F(s,\cdot),\pd_t v_m(s,\cdot) \rangle_{L^2(\Omega)},\ s \in (0,T).
$$
Now, by integrating with respect to $s$ over $(0,t)$, we obtain that
\bel{e2}
a(t;v_m(t,\cdot),v_m(t,\cdot)) = \int_0^t a'(s;v_m(s,\cdot),v_m(s,\cdot) ) ds - 2 \re \int_0^t \langle F(s,\cdot),\pd_t v_m(s,\cdot) \rangle_{L^2(\Omega)} ds. 
\ee
Next, as $\int_0^t \langle F(s,\cdot),\pd_t v_m(s,\cdot) \rangle_{L^2(\Omega)} ds = \langle F(t,\cdot),v_m(t,\cdot) \rangle_{L^2(\Omega)} -
\int_0^t \langle \pd_t F(s,\cdot),v_m(s,\cdot) \rangle_{L^2(\Omega)} ds$, we get
$$ \left| \int_0^t \langle F(s,\cdot),\pd_t v_m(s,\cdot) \rangle_{L^2(\Omega)} ds \right| \leq \| F(t,\cdot) \|_{L^2(\Omega)} \|  v_m(t,\cdot) \|_{L^2(\Omega)} + \int_0^t \| \pd_t F(s,\cdot) \|_{L^2(\Omega)} \|  v_m(s,\cdot) \|_{L^2(\Omega)} ds,$$
so that $\| v_m(t,\cdot) \|_{H^1(\Omega)}^2$ can be upper bounded with the aid of  \eqref{ell1}, \eqref{e1} and \eqref{e2},  by
the following expression
$$c \left( \int_0^t \| v_m(s,\cdot) \|_{H^1(\Omega)}^2 ds + \int_0^t \| \pd_t F(s,\cdot) \|_{L^2(\Omega)}^2 ds \right) + \frac{2}{\alpha} \| F(t,\cdot) \|_{L^2(\Omega)} \|  v_m(t,\cdot) \|_{L^2(\Omega)} 
+ \frac{\lambda}{\alpha} \| v_m(t,\cdot) \|_{L^2(\Omega)}^2,
$$
where $c$ is a positive constant depending only on $M$. From this, the two basic inequalities
$$ \| F(t,\cdot) \|_{L^2(\Omega)} \|  v_m(t,\cdot) \|_{L^2(\Omega)} \leq \frac{\alpha}{4} \|  v_m(t,\cdot) \|_{L^2(\Omega)}^2 + \frac{1}{\alpha} \| F(t,\cdot) \|_{L^2(\Omega)}^2 $$
and
$$ \| F(t,\cdot) \|_{L^2(\Omega)}^2 = 2 \re \int_0^t \langle F(s,\cdot) , \pd_t F(s,\cdot) \rangle_{L^2(\Omega)} ds \leq \int_0^t \left( \norm{F(s,\cdot)}_{L^2(\Omega)}^2 + \norm{\pd_t F(s,\cdot)}_{L^2(\Omega)}^2 \right) ds, $$
and from the estimate \eqref{e1b}, it then follows that
$$
\| v_m(t,\cdot) \|_{H^1(\Omega)}^2 \leq  2 c  \int_0^t \| v_m(s,\cdot) \|_{H^1(\Omega)}^2 ds + 2 \left( c +\frac{2}{\alpha^2}  + \frac{ c_0^2}{4}+\frac{\lambda c_0^2}{\alpha}\right) \| F \|_{H^1(0,T;L^2(\Omega))}^2.
$$
Then, an application of Gronwall's lemma provides a constant $C=C(T,M,\alpha) \in (0,+\infty)$ such that
\bel{l1b}
\norm{v_m}_{L^\infty(0,T;H^1(\Omega))}\leq C\norm{F}_{H^1(0,T;L^2(\Omega))}.
\ee

c) Further, we put $p(t,x):=|A(t,x)|^2+q(t,x)$ for $(t,x) \in Q$, integrate by parts in the first integral of \eqref{e1} and obtain for all $u, v \in H_0^1(\Omega)$ that
\beas
a'(t;u,v)&=& i \int_\Omega \left( u(x) \pd_t A(t,x) \cdot \overline{\nabla v(x)} - \pd_t A(t,x) \cdot \nabla u(x) \overline{v(x)} \right) dx+\int_\Omega \pd_t p(t,x )u(x) \overline{v(x)}dx\\
&=& \int_\Omega \left( -2i \pd_t A(t,x) \cdot \nabla u(x)+(\pd_t p(t,x) - i \nabla \cdot \pd_t A(t,x) ) u(x) \right) \overline{v(x)}dx.
\eeas
This and \eqref{od2} yield 
\bel{od3}
\left\{\begin{array}{ll}
i \langle \pd_t w_m(t,\cdot),e_k \rangle_{L^2(\Omega)} -a(t;w_m(t,\cdot),e_k) = \langle  F_m(t,\cdot) , e_k \rangle_{L^2(\Omega)}, & t \in (0,T), \\
w_m(0) = 0, &
\end{array}
\right. 
\ee
for all $k=1,\ldots,m$, where
\bel{od4}
F_m(t,x) := -2i \pd_t A(t,x) \cdot \nabla v_m(t,x)+(\pd_t p(t,x) - i \nabla \cdot \pd_t A(t,x) ) v_m(t,x) + \pd_t F(t,x),\ (t,x) \in Q. 
\ee 
Next, multiplying the first line in \eqref{od3} by $g_{k,m}'(t)$ and summing up over $k=1,\ldots,m$, lead to
$$ 
i \langle \pd_t w_m(t,\cdot), w_m(t,\cdot) \rangle_{L^2(\Omega)} -a(t;w_m(t,\cdot),w_m(t,\cdot)) = \langle  F_m(t,\cdot) , w_m(t,\cdot) \rangle_{L^2(\Omega)},\  t \in (0,T).
$$
Therefore, we have
$$ \frac{d}{d s} \norm{w_m(s,\cdot)}_{L^2(\Omega)}^2 = 2 \im \langle  F_m(s,\cdot) , w_m(s,\cdot) \rangle_{L^2(\Omega)},\ s \in (0,T). $$
Now, for each $t \in (0,T)$, we find upon integrating both sides of the above identity over $(0,t)$ that
$$ \norm{w_m(t,\cdot)}_{L^2(\Omega)}^2 \leq \int_0^t \norm{w_m(s,\cdot)}_{L^2(\Omega)}^2 ds + \int_0^t \norm{F_m(s,\cdot)}_{L^2(\Omega)}^2 ds, $$
which, combined with \eqref{od4}, entails
$$ \norm{w_m(t,\cdot)}_{L^2(\Omega)}^2 \leq \int_0^t \norm{w_m(s,\cdot)}_{L^2(\Omega)}^2 ds + c \left( \norm{\pd_t F}_{L^2(Q)}^2 + \norm{v_m}_{L^\infty(0,T;H^1(\Omega))}^2 \right),$$
for some constant $c=c(T,M) \in (0,+\infty)$. Therefore, we have
$$ 
\norm{w_m(t,\cdot)}_{L^2(\Omega)}^2 \leq c \left(  \norm{\pd_t F}_{L^2(Q)}^2 + \norm{v_m}_{L^\infty(0,T;H^1(\Omega))}^2 \right),\ t \in (0,T), 
$$
by Gronwall's lemma, and consequently
\bel{e3}
\norm{w_m}_{L^\infty(0,T;L^2(\Omega))}\leq C \norm{F}_{H^1(0,T;L^2(\Omega))},
\ee
by \eqref{l1b}, where $C$ is another positive constant depending only on $T$, $M$ and $\alpha$. \\

\noindent 2) Having established \eqref{l1b} and \eqref{e3}, we turn now to showing existence of a solution to \eqref{eq2}.
This can be done in accordance with \eqref{l1b} by extracting a subsequence $(v_{m'})_{m'}$ of $(v_m)_m$, which converges to $v \in L^\infty(0,T;H_0^1(\Omega))$ in the weak-star topology of $L^\infty(0,T;H_0^1(\Omega))$. By substituting $m'$ for $m$ in \eqref{od1} and sending $m'$ to infinity, we get that 
\bel{e4}
\left\{\begin{array}{ll}
(i \pd_t + \Delta_A + q) v = F & \mbox{in}\ Q, \\
v(0,\cdot)=0 & \mbox{in}\ \Omega.
\end{array}
\right.
\ee
As a consequence, we have $\pd_t v \in L^\infty(0,T;H^{-1}(\Omega))$ and hence $v  \in L^\infty(0,T;H_0^1(\Omega)) \cap W^{1,\infty}(0,T;H^{-1}(\Omega))$. Further, due to \eqref{e3} and the Banach-Alaoglu theorem, there exists a subsequence of $(w_m)_m$ which converges to $w \in L^\infty(0,T;L^2(\Omega))$ in the weak-star topology of $L^\infty(0,T;L^2(\Omega))$. Since $w_m=\pd_t v_m$ for every $m \in \mathbb{N}^*$ then we necessarily have $\pd_t v = w \in L^\infty(0,T;L^2(\Omega))$ and thus $v\in L^\infty(0,T;H^1_0(\Omega)) \cap W^{1,\infty}(0,T;L^2(\Omega))$. Further, by arguing in the exact same way as in the derivation of \cite[Theorem 8.3 and Remark 10.2, Chapter 3]{LM1}, we get that $v \in \cC([0,T];H^1_0(\Omega))\cap \cC^1([0,T];L^2(\Omega))$. Moreover, for all fixed $t\in[0,T]$, we deduce from \eqref{e4} that $v(t,\cdot)$ is solution to the elliptic boundary value problem
$$\left\{\begin{array}{ll}
\left( \Delta_{A} + q(t,.) \right )v(t,\cdot) = F(t,\cdot)-i\pd_t v(t,\cdot) & \mbox{in}\ \Omega,\\ 
v(t,\cdot)=0, &\mbox{on}\ \pd\Omega.
\end{array}
\right.$$
As $F-i\pd_t v \in \mathcal C([0,T];L^2(\Omega))$ then we have $v\in \cC([0,T],H_0^1(\Omega)\cap H^2(\Omega) )$ and
\eqref{l1a} follows directly from this, \eqref{ell0b}, \eqref{l1b} and \eqref{e3}.
\end{proof}

\begin{remark}
\label{rmk-p}
\begin{enumerate}[a)]
\item With reference to \eqref{e1b}, we point out for further use that the solution $v$ to \eqref{l1} satisfies the estimate
$$ \| v \|_{L^\infty(0,T;L^2(\Omega))} \leq c_0 \| F \|_{L^2(Q)}, $$
for some constant $c_0$ depending only on $\Omega$, $T$ and $M$. 
\item Let $M$, $A$ and $q$ be the same as in Lemma \ref{l1}, and let $F \in H^1(0,T;L^2(\Omega))$ satisfy $F(T,\cdot)=0$. Putting $\tilde{A}(t,x) :=-A(T-t,x)$, $\tilde{q}(t,x):=q(T-t,x)$ and $\tilde{F}(t,x):=\overline{F(T-t,x)}$ for $(t,x) \in Q$, we see that we have
\bel{eq2b}
\left\{\begin{array}{ll}
(i\pd_t +\Delta_{A} + q ) v = F & \mbox{in}\ Q,\\ 
 v(T,\cdot)=0 &\mbox{in}\ \Omega,\\ 
 v=0 &\mbox{on}\ \Sigma,
\end{array}
\right.
\ee
if and only if $\tilde{v}(t,x):=\overline{v(T-t,x)}$ is a solution to the system \eqref{eq2} where $(\tilde{A},\tilde{q},\tilde{F})$ is substituted for $(A,q,F)$. Therefore, by Lemma \ref{l1}, there exists a unique solution $v \in \cC([0,T],H_0^1(\Omega) \cap H^2(\Omega)) \cap \cC^1([0,T],L^2(\Omega))$ to \eqref{eq2b}, and it is clear that $v$ verifies the estimate \eqref{l1a}.
\end{enumerate}
\end{remark}

\subsection{Completion of the proof of Proposition \ref{p1}} 
\label{sec-p1}
In light of  \cite[Theorem 2.3, Chapter 4]{LM2} there exists $G \in H^{3,2}(Q)$ satisfying
$$G(0,\cdot)=\pd_t G(0,\cdot)=0\ \mbox{in}\ \Omega\ \mbox{and}\ G=g\ \mbox{on}\ \Gamma,$$
and
\bel{p1c} 
\norm{G}_{H^{3,2}(Q)}\leq C\norm{g}_{\cH(\Sigma)}, 
\ee
for some positive constant $C$, depending only on $\Omega$ and $T$. Therefore, the function
\bel{e5}
F:=-(i\pd_t+\Delta_{A}+q)G \in H^1(0,T;L^2(\Omega))
\ee 
verifies $F(0,\cdot)=0$ in $\Omega$. Let $v$ be the $\cC([0,T],H_0^1(\Omega)\cap H^2(\Omega) )\cap \cC^1([0,T],L^2(\Omega))$-solution to \eqref{eq2}, associated with the source term $F$ defined by \eqref{e5}, which is given by Lemma \ref{l1}.
Then, $u:=G+v\in H^{1,2}(Q)$ is a solution to \eqref{1.1} and \eqref{p1b} follows directly from \eqref{l1a} and \eqref{p1c}.
Finally, we get that such a solution is unique by applying \eqref{p1b} with $g=0$.

\section{GO solutions}
\setcounter{equation}{0}
\label{sec-GO}
In this section we build appropriate geometric optics (GO) solutions to the magnetic Schr\"odinger equation in $Q$, which are used in the derivation of the stability estimates of Theorem \ref{t1}, presented in Sections \ref{sec-sema} and \ref{sec-seel}.

Namely, given $A_j\in W^{5,\infty}(Q)^n\cap H^6(Q)^n$ and $q_j\in W^{4,\infty}(Q)$, $j=1,2$, fulfilling the conditions \eqref{cond1}-\eqref{cond3}, we seek a solution $u_j$ to the magnetic Schr\"odinger equation
\bel{GO1}
(i\pd_t +\Delta_{A_j} +q_j ) u_j=0\ \mbox{in}\ Q,
\ee
of the form
\bel{GO2} 
u_j(t,x)= \varphi_{\sigma}(t,x) \left(u_{j,1}(t,x)+ \sigma^{-1} u_{j,2}(t,x)  \right)+r_{j,\sigma}(t,x)\ \mbox{with}\ \varphi_{\sigma}(t,x):=e^{i\sigma(-\sigma t+x\cdot\omega)}.
\ee
Here $\sigma \in (1,+\infty)$ and $\omega \in \mathbb S^{n-1} :=\{y\in\R^n;\ |y|=1\}$ are arbitrarily fixed and the remainder term $r_{j,\sigma}$ in the asymptotic expansion of $u_j$ with respect to $\sigma^{-1}$, scales at most like $\sigma^{-1}$ as $\sigma\to+\infty$, in a sense that we will make precise below. Moreover, we impose that $u_{j,1}$ and $u_{j,2}$ be in $H^{3}(Q)$, and that they satisfy 
\bel{GO3}
\omega \cdot \nabla_{A_j} u_{j,1}  = 0\ \mbox{and}\ 2i  \omega \cdot \nabla_{A_j} u_{j,2} + (i\pd_t +\Delta_{A_j} +q_j )u_{j,1}=0\
\mbox{in}\ Q,
\ee
and
\bel{GO5} 
\exists \delta \in \left( 0, \frac{T}{4} \right),\ \left( t \in (0,\delta) \cup (T-\delta,T) \right) \Rightarrow \left( u_{j,1}(t,x)=u_{j,2}(t,x)=0,\ x \in \Omega \right).
\ee
The two conditons in \eqref{GO3} can be understood from the formal commutator formula
$[ i \pd_t + \Delta_{A_j} , \varphi_{\sigma} ]= i \pd_t \varphi_{\sigma} + \Delta \varphi_{\sigma} + 2 \nabla \varphi_{\sigma} \cdot \nabla_{A_j} = 2 i \sigma \varphi_{\sigma}  \omega \cdot \nabla_{A_j}$, entailing
\beas
& & (i\pd_t+\Delta_{A_j}+q_j) (u_j-r_{j,\sigma}) 
=  (i\pd_t+\Delta_{A_j}+q_j)  \varphi_{\sigma} ( u_{j,1} + \sigma^{-1} u_{j,2} )\\
& = & \varphi_{\sigma} \left( 2 i \sigma \omega \cdot \nabla_{A_j} u_{j,1} + (i\pd_t+\Delta_{A_j}+q_j) u_{j,1} + 2i \omega \cdot \nabla_{A_j} u_{j,2} + \sigma^{-1} (i \pd_t +\Delta_{A_j}+q_j)u_{j,2} \right)
\eeas
in $Q$.
This and \eqref{GO1}-\eqref{GO2} lead to defining $r_{1,\sigma}$ by
\bel{GO61}
\left\{\begin{array}{ll}
(i\pd_t +\Delta_{A_1} + q_1) r_{1,\sigma} = -\sigma^{-1} \varphi_{\sigma} (i\pd_t+\Delta_{A_1}+q_1)u_{1,2} & \mbox{in}\ Q,\\ 
r_{1,\sigma}(0,\cdot)=0 &\mbox{in}\ \Omega,\\ 
r_{1,\sigma}=0 &\mbox{on}\ \Sigma,
\end{array}
\right.
\ee
and $r_{2,\sigma}$ by
\bel{GO62}
\left\{\begin{array}{ll}
(i\pd_t +\Delta_{A_2} + q_2) r_{2,\sigma} = -\sigma^{-1} \varphi_{\sigma} (i\pd_t+\Delta_{A_2}+q_2)u_{2,2} & \mbox{in}\ Q,\\ 
r_{2,\sigma} (T,\cdot)=0 &\mbox{in}\ \Omega,\\ 
r_{2,\sigma} =0 &\mbox{on}\ \Sigma.
\end{array}
\right.
\ee
The initial condition in \eqref{GO61} and the final condition in \eqref{GO62} are imposed in such a way that the product $r_{1,\sigma} \overline{r_{2,\sigma}}$ vanishes at both ends of the time interval $(0,T)$.

The first step of the construction of the functions $u_{j,k}$, for $j,k=1,2$, involves extending the two magnetic potentials $A_1$ and $A_2$ to $(0,T) \times \R^n$ as follows. First, we refer to \cite[ Theorem 5 in Section 3]{St} and pick a magnetic potential $\tilde{A}_1\in W^{5,\infty}((0,T)\times\R^n;\R)^n \cap H^6((0,T)\times\R^n;\R)^n$ which coincides with $A_1$ in $Q$ and satisfies
$$ \exists R \in (0,+\infty),\ \forall t \in [0,T],\ \supp\ \tilde{A}_1(t,\cdot) \subset  \{x\in\R^n,\ |x| \leq R\} $$
and the estimate
\bel{g0} 
\norm{\tilde{A}_1}_{W^{5,\infty}((0,T)\times\R^n)^n}\leq C\norm{A_1}_{W^{5,\infty}(Q)^n}\ \mbox{and}\ \norm{\tilde{A}_1}_{H^6((0,T)\times\R^n)^n}\leq C\norm{A_1}_{H^6(Q)^n},
\ee
for some positive constant $C$ depending only on $T$ and $\Omega$. 
Thus, putting
\bel{A2}
\tilde{A}_2(t,x) :=
\left\{
\begin{array}{ll} 
A_2(t,x) & \mbox{if}\ x \in \Omega,\\ \tilde{A}_1(t,x) & \mbox{if}\ x \in \R^n \setminus \Omega, 
\end{array}
\right.
\ee
we infer from \eqref{cond1} that $\tilde{A}_2\in W^{5,\infty}((0,T)\times\R^n)^n\cap H^6((0,T)\times\R^n)^n$. Moreover, it is clear from \eqref{g0}-\eqref{A2} upon possibly substituting $\max(1,C)$ for $C$ in \eqref{g0}, that  
\bel{es}
\norm{\tilde{A}_j}_{W^{5,\infty}((0,T)\times\R^n)^n}\leq C\max_{k=1,2}\norm{A_k}_{W^{5,\infty}(Q)^n}\ \mbox{and}\ \norm{\tilde{A}_j}_{H^6((0,T)\times\R^n)^n}\leq C\max_{k=1,2}\norm{A_k}_{H^6(Q)^n},\ j=1,2. \ee

The next step is to introduce two functions, the first one $\chi=\chi_\delta \in \mathcal C^\infty(\R;[0,1])$, being supported in
$(\delta,T-\delta)$, satisfies $\chi(t)=1$ if $t \in [2\delta,T-2\delta]$ and fulfills
$$ \forall k \in {\mathbb N},\ \exists C_k \in (0,+\infty),\ \norm{\chi}_{W^{k,\infty}(\R)} \leq C_k \delta^{-k}, $$
whereas the second one is defined for $\tau \in \R$, $\xi \in \omega^{\bot}:=\{ x\in\R^n;\ x \cdot \omega=0 \}$ and $y\in \mathbb S^{n-1}\cap\omega^{\bot}$, by
\bel{GO8}
\beta(t,x) :=y \cdot \nabla \left( e^{-i(t \tau + \xi\cdot x)} \exp \left( -i \int_\R A(t,x+s\omega) \cdot \omega ds \right) \right),\ (t,x) \in (0,T) \times \R^n.
\ee
Here we have set $A:=\tilde{A}_1-\tilde{A}_2$ (that is $A=A_1-A_2$ in $Q$ and $A=0$ in $\left( (0,T)\times\R^n \right) \setminus Q$) in such a way that we have $\omega \cdot \nabla \beta(t,x)=0$ for all $(t,x) \in (0,T) \times \R^n$. 

Now, a direct calculation shows that each of the two functions
$$
u_{1,1}(t,x)  := \chi(t) \beta(t,x) \exp \left(i\int_0^{+\infty}\tilde{A}_1(t,x+s\omega)\cdot\omega ds\right)
$$
and
$$
u_{2,1}(t,x) := \chi(t) \exp\left(i\int_0^{+\infty}\tilde{A}_2(t,x+s\omega)\cdot\omega ds\right)
$$
is a solution to the first equation of \eqref{GO3} satisfying the condition \eqref{GO5}. Further, it follows from this,\eqref{cond3} and \eqref{es}-\eqref{GO8} that
\bel{GO101}
\norm{u_{1,1}}_{H^3(Q)} \leq C \langle \tau,\xi \rangle^4 \delta^{-3}\ \mbox{and}\ \norm{u_{2,1}}_{H^3(Q)} \leq C\delta^{-3},
\ee
where $C$ denotes a positive constant depending only $\Omega$, $T$ and $M$, which may change from line to line, and $\langle \tau,\xi \rangle$ is a shorthand for $(1+\tau^2+\xi^2)^{1 \slash 2}$.

Similarly, using that any $x\in\R^n$ decomposes into the sum $x=x_\bot+s\omega$ with $s:=x\cdot\omega$ and $x_\bot:=x-s\omega\in\omega^\bot$, it can be checked through standard computations that
$$
u_{j,2}(t,x_\bot+s\omega) :=-\frac{1}{2i}\int_0^s \exp \left( -i\int_{s_1}^{s}\tilde{A}_j(t,x_\bot+s_2\omega) \cdot \omega d s_2 \right) (i\pd_t+\Delta_{A_j}+q_j)u_{j,1}(t,x_\bot+s_1\omega) d s_1,\ j=1,2,
$$
is a solution to the second equation of \eqref{GO3} obeying the condition \eqref{GO5}.
Further, by \eqref{cond3} and \eqref{es} we have
\bel{GO102}
\norm{u_{1,2}}_{H^3(Q)}\leq C \langle \tau,\xi \rangle^6 \delta^{-4}\ \mbox{and}\ \norm{u_{2,2}}_{H^3(Q)}\leq C \delta^{-4},
\ee
and
\bel{GO102b}
\norm{(i\pd_t +\Delta_{A_1} + q_1)u_{1,2}}_{L^2(Q)} \leq C \langle \tau,\xi \rangle^5 \delta^{-2}\ \mbox{and}\ 
\norm{(i\pd_t +\Delta_{A_2} + q_2)u_{2,2}}_{L^2(Q)}\leq C\delta^{-2}.
\ee

Having specified $u_{j,k}$ for $j, k=1,2$, we turn now to examining the remainder term $r_{j,\sigma}$. 
We first infer from Lemma \ref{l1} (resp., Statement b) of Remark \ref{rmk-p}) that $r_{1,\sigma}$ (resp. $r_{2,\sigma}$) is well defined as the $\cC([0,T],H_0^1(\Omega)\cap H^2(\Omega) )\cap \cC^1([0,T],L^2(\Omega))$-solution to \eqref{GO61} (resp., \eqref{GO62}). Next, Statement a) in Remark \ref{rmk-p} and \eqref{GO102b} yield
\bel{RG}
\norm{r_{1,\sigma}}_{L^2(Q)}\leq C \norm{(i\pd_t +\Delta_{A_1} + q_1 )u_{1,2}}_{L^2(Q)} \sigma^{-1} \leq C \langle \tau,\xi \rangle^5 \delta^{-2} \sigma^{-1}.
\ee
On the other hand, we know from \eqref{l1a} and \eqref{GO102} that
$$
\norm{r_{1,\sigma}}_{L^2(0,T;H^2(\Omega))} \leq C \norm{e^{i\sigma(-\sigma t+x\cdot\omega)}(i\pd_t +\Delta_{A_1} + q_1)u_{1,2}}_{H^1(0,T;L^2(\Omega))} \sigma^{-1} \leq C \langle \tau,\xi \rangle^6 \delta^{-3} \sigma,
$$
Thus, interpolating with \eqref{RG}, we have $\norm{r_{1,\sigma}}_{L^2(0,T;H^1(\Omega))}\leq C \langle \tau,\xi \rangle^6 \delta^{-3}$ and hence
\bel{GO111} 
\norm{r_{1,\sigma}}_{L^2(0,T;H^1(\Omega))}+\sigma\norm{r_{1,\sigma}}_{L^2(Q)}\leq C \langle \tau,\xi \rangle^6 \delta^{-3}.
\ee
Analogously, we establish that
\bel{GO112} 
\norm{r_{2,\sigma}}_{L^2(0,T;H^1(\Omega))}+\sigma\norm{r_{2,\sigma}}_{L^2(Q)}\leq C\delta^{-3}.
\ee
Having built $u_{j,k}$ and $r_{j,\sigma}$, for $j,k=1,2$, fulfilling \eqref{GO1}--\eqref{GO5}, we are now in position to derive the stability estimates \eqref{t1a}-\eqref{t1b} of Therorem \ref{t1}.

\section{Proof of the stability estimate \eqref{t1a}}
\setcounter{equation}{0}
\label{sec-sema}

We stick to the notations of Section \ref{sec-GO} and recall from \eqref{GO5} that
\bel{GO12}
u_1(0,x)=u_2(T,x)=0,\ x\in \Omega.
\ee
The proof of \eqref{t1a} boils down to a suitable estimate of the Fourier transform of the function $\chi^2 A$, presented in Lemma
\ref{l3}.

\subsection{Estimation of the Fourier transform of $\chi^2 A$}
We start by proving the following technical estimate.
\begin{lemma}
\label{l2} 
There exists a constant $C=C(T,\Omega,M) \in (0,+\infty)$ such that we have
\bea
& & \abs{\int_{\R^{n+1}} \chi^2(t) \beta(t,x) e^{i\int_0^{+\infty}A(t,x+s\omega)\cdot\omega ds} A(t,x)\cdot\omega dxdt} \nonumber \\
&\leq & C \left( \langle \tau,\xi \rangle^6  \delta^{-8}  \sigma^5 \norm{\Lambda_{A_1,q_1}-\Lambda_{A_2,q_2}}  +  \langle \tau,\xi \rangle^8 \delta^{-6} \sigma^{-1} \right), \label{l2a}
\eea
uniformly in $\xi \in \omega^\perp$.
\end{lemma}
\begin{proof} 
For $j=1,2$, put $\psi_{j,\sigma}:=u_j-r_{j,\sigma}=u_j$ in $\Sigma$ and let $v_2$ be the $H^{1,2}(Q)$-solution to
$$
\left\{
\begin{array}{ll}
(i\pd_t +\Delta_{A_2}+q_2) v_2=0 & \mbox{in}\ Q,\\
v_2(0,\cdot)=0 & \mbox{in}\ \Omega,\\
v_2=\psi_{1,\sigma} & \mbox{on}\ \Sigma,
\end{array}
\right.
$$
given by Proposition \ref{p1}. In light of \eqref{GO1} the function $w:=v_2-u_1$ then solves
$$
\left\{
\begin{array}{ll}
(i\pd_t +\Delta_{A_2}+q_2) w= 2iA\cdot \nabla u_1+ V u_1 & \mbox{in}\ Q,\\
w(0,\cdot)=0 & \mbox{in}\ \Omega,\\
w=0 & \mbox{on}\ \Sigma,
\end{array}
\right.$$
with $V: =i \nabla \cdot A -(|A_1|^2-|A_2|^2)+q_1-q_2$. Next, by multiplying the first equation of the above system by $\overline{u_2}$ and integrating by parts over $Q$, we deduce from \eqref{GO1} and \eqref{GO12} that
\bel{l2c}
\int_Q (2i A \cdot \nabla + V)u_1(t,x) \overline{u_2(t,x)} dxdt=\int_{\Sigma} \pd_\nu w(t,x) \overline{u_2(t,x)} d\sigma(x)dt.
\ee
Further, since $(\pd_\nu + i A_2 \cdot \nu) v_2 = \Lambda_{A_2,q_2} \psi_{1,\sigma}$ and
$(\pd_\nu  + i A_1 \cdot \nu) u_1 = \Lambda_{A_1,q_1} \psi_{1,\sigma}$, we have $\pd_\nu w = (\Lambda_{A_2,q_2}-\Lambda_{A_1,q_1}) \psi_{1,\sigma}$ in virtue of \eqref{cond1}, and hence 
\bea
\abs{\int_{\Sigma}\pd_\nu w(t,x) \overline{u_2(t,x)}d\sigma(x)dt} & \leq & 
\norm{(\Lambda_{A_2,q_2}-\Lambda_{A_1,q_1})\psi_{1,\sigma}}_{L^2(\Sigma)} \norm{\psi_{2,\sigma}}_{L^2(\Sigma)} \nonumber \\
&\leq & C\norm{\Lambda_{A_1,q_1}-\Lambda_{A_2,q_2}}\norm{\psi_{1,\sigma}}_{\cH(\Sigma)}\norm{\psi_{2,\sigma}}_{L^2(\Sigma)} \nonumber \\
&\leq & C\norm{\Lambda_{A_2,q_2}-\Lambda_{A_2,q_2}}  \langle \tau,\xi \rangle^6 \delta^{-8} \sigma^6, \label{l2d}
\eea
by Corollary \ref{cor}, the continuity of the trace operator from $H^3(Q)$ into $\cH(\Sigma)$, \eqref{GO2} and the estimates \eqref{GO101}-\eqref{GO102}.  

On the other hand, we know from \eqref{GO2} that
\bea
\int_Q (2iA\cdot\nabla +V)u_1 \overline{u_2}(t,x)dxdt
& = & -2 \sigma \int_Q (A \cdot \omega) u_{1,1}\overline{u_{2,1}}(t,x)dxdt + r \nonumber \\
& = & -2 \sigma \int_Q \chi^2(t)\beta(t,x)e^{i\int_0^{+\infty}A(t,x+s\omega)\cdot\omega ds} A(t,x)\cdot\omega dxdt + r, \label{l2e}
\eea
where
\beas
r & := & -2 \sigma \int_Q u_{1,1} \left( \overline{u_{2,2}} \sigma^{-1} + \varphi_{\sigma} \overline{r_{2,\sigma}} \right) A(t,x) \cdot \omega dx dt \nonumber \\
& & -2  \int_Q A \cdot \left( \varphi_{\sigma} \left( \nabla u_{1,1} + \nabla u_{1,2}  \sigma^{-1} \right) + \nabla r_{1,\sigma} \right) \overline{u_2}(t,x) dx dt + \int_Q V u_1 \overline{u_2}(t,x) dx dt.
\eeas
Since $\abs{r} \leq C \langle \tau,\xi \rangle^8 \delta^{-6}$ by \eqref{GO111}-\eqref{GO112}, it then follows from \eqref{l2e} that
\beas
& & \sigma^{-1} \abs{\int_Q (2iA\cdot\nabla +V)u_1\overline{u_2}(t,x) dxdt}\\
& \geq & 2 \abs{\int_Q \chi^2(t) \beta(t,x)e^{i\int_0^{+\infty}A(t,x+s\omega)\cdot\omega ds} A(t,x) \cdot\omega dxdt}
-C \langle \tau,\xi \rangle^8 \delta^{-6} \sigma^{-1},
\eeas
which, combined with \eqref{l2c}-\eqref{l2d}, yields \eqref{l2a}.
\end{proof}

Having established Lemma \ref{l2} we may now estimate the Fourier transform of $\chi A$.
We recall that the Fourier transform of a function $f \in L^1(\R^{1+n})^n$ is defined for all $(\tau,\xi) \in \R \times \R^n$ by
$$ \hat{f}(\tau,\xi) :=(2\pi)^{-{n+1\over2}}\int_{\R^{1+n}}e^{-i(t\tau+x\cdot\xi)} f(t,x)dxdt.$$

\begin{lemma}
\label{l3} 
There exists a positive constant $C$, depending only on $T$, $\Omega$ and $M$, such that the inequality
\bel{l3a}
|\xi| \abs{\widehat{\chi^2A}(\tau,\xi)}\leq C \left( \langle \tau,\xi \rangle^6  \delta^{-8} \sigma^5 \norm{\Lambda_{A_1,q_1}-\Lambda_{A_2,q_2}} + \langle \tau,\xi \rangle^8 \delta^{-6} \sigma^{-1} \right),
\ee
holds for any $(\tau,\xi)\in \R^{1+n}$.
\end{lemma}
\begin{proof}
The estimate \eqref{l3a} being obviously true for $\xi =0$, we will solely focus on the case $\xi\neq 0$.
We use the decomposition $x=x_\bot+ \kappa \omega$, where $\kappa:=x\cdot\omega$ and $x_\bot :=x- \kappa \omega$, and recall from \eqref{GO8} that we have $\beta(t,x)=\beta(t,x_\bot)$, so we obtain
\bea
& & \int_{\R^{1+n}} \chi^2(t) \beta(t,x) e^{i\int_0^{+\infty}A(t,x+s\omega)\cdot\omega ds} A(t,x)\cdot\omega dxdt \nonumber \\
& = &  \int_\R \int_\R \int_{\omega^\bot} \chi^2(t) \beta(t,x_\bot) e^{i\int_{\kappa}^{+\infty} A(t,x_\bot+s\omega)\cdot\omega ds} A(t,x_\bot+\kappa\omega)\cdot\omega  dx_\bot d\kappa dt \nonumber \\
& = & i \int_\R \int_{\omega^\bot} \chi^2(t) \beta(t,x_\bot) \left( \int_\R \pd_{\kappa} e^{i\int_{\kappa}^{+\infty} A(t,x_\bot+s\omega)\cdot\omega ds} d\kappa \right)  dx_\bot  dt \nonumber \\
& = & i \int_\R \chi^2(t) \left( \int_{\omega^\bot} \beta(t,x_\bot) \left(1- e^{i\int_\R A(t,x_\bot+s\omega)\cdot\omega ds} \right) dx_\bot\right) dt \nonumber \\
&= & i \int_\R\chi^2(t)e^{-it\tau}\left(\int_{\omega^\bot}y\cdot\nabla \left( e^{-i\xi\cdot x_\bot}e^{-i\int_\R A(t,x_\bot+s\omega)\cdot\omega ds}\right)\left(1- e^{i\int_\R A(t,x_\bot+s\omega)\cdot\omega ds}\right)dx_\bot\right)dt. \label{q1}
\eea
Next, as we have
\beas
& & \int_{\omega^\bot}y\cdot\nabla \left( e^{-i\xi\cdot x_\bot}e^{-i\int_\R A(t,x_\bot+s\omega)\cdot\omega ds}\right)\left(1- e^{i\int_\R A(t,x_\bot+s\omega)\cdot\omega ds}\right)dx_\bot \\
& = & \int_{\omega^\bot} \nabla \cdot \left( y e^{-i\xi\cdot x_\bot}e^{-i\int_\R A(t,x_\bot+s\omega)\cdot\omega ds}\right) \left(1- e^{i\int_\R A(t,x_\bot+s\omega)\cdot\omega ds}\right)dx_\bot \\
& = & -i \int_{\omega^\bot} e^{-i\xi\cdot x_\bot} y \cdot \nabla \left( \int_\R  A(t,x_\bot+s\omega)\cdot\omega ds \right)
d x_\bot,
\eeas
by integrating by parts, \eqref{q1} and the Fubini theorem entail
\beas
& &  \int_{\R^{1+n}} \chi^2(t) \beta(t,x) e^{i\int_0^{+\infty}A(t,x+s\omega)\cdot\omega ds} A(t,x)\cdot\omega dxdt\\
& = & \int_\R \chi^2(t) e^{-it\tau}\left(\int_{\omega^\bot}\left(\int_\R y\cdot \nabla (A(t,x_\bot+s\omega)\cdot\omega) ds\right)e^{-i\xi\cdot x_\bot}dx_\bot\right)dt\\
& = &  \int_\R\chi^2(t)e^{-it\tau} \left( \int_{\R^n} e^{-i\xi\cdot x} y \cdot \nabla (A(t,x)\cdot\omega)  dx\right)dt\\
& = &  \int_{\R^{n+1}}  e^{-i(t\tau + x \cdot \xi)} \chi^2(t)  y \cdot  \nabla (A(t,x) \cdot\omega) dxdt.
\eeas
Therefore, taking $y=\frac{\xi}{|\xi|}$ and applying Stokes formula to the above integral, we obtain
\beas
& & i \int_{\R^{1+n}} \chi^2(t) \beta(t,x) e^{i\int_0^{+\infty} A(t,x+s\omega)\cdot\omega ds} A(t,x)\cdot\omega dxdt \\
&=& -|\xi| \int_{\R^{n+1}} e^{-i(t\tau+ x \cdot \xi)}\chi^2(t) A(t,x)\cdot\omega dxdt\\
&= & - (2\pi)^{\frac{n+1}{2}} |\xi|\widehat{\chi^2A}(\tau,\xi)\cdot\omega,
\eeas
which yields
\bel{l3b}
|\xi| \abs{\widehat{\chi^2A}(\tau,\xi)\cdot\omega} \leq 
C \langle \tau,\xi \rangle^6 \delta^{-8} \left( \norm{\Lambda_{A_2,q_2}-\Lambda_{A_2,q_2}}\sigma^5 + \langle \tau,\xi \rangle ^2 \delta^{2}\sigma^{-1} \right),
\ee
in virtue of \eqref{l2a}.

Further, since $\nabla \cdot A=0$ in $Q$, by \eqref{cond2}, then we have $\widehat{\chi^2A}(\tau,\xi)\cdot\xi=0$ by direct calculation and hence
$$ \widehat{\chi^2A}(\tau,\xi)=\sum_{k=1}^{n-1} \left( \widehat{\chi^2A}(\tau,\xi)\cdot e_k \right) e_k, $$
for any orthonormal basis $\{ e_1,\ldots,e_{n-1} \}$ of $\xi^\bot$.
Finally, \eqref{l3a} follows directly from this upon applying \eqref{l3b} with $\omega=e_k$ for $k=1,\ldots,n-1$.
\end{proof}

Having established Lemma \ref{l3}, we are now in position to derive the stability estimate \eqref{t1a}.

\subsection{Completion of the proof}
\label{sec-completion}

We start by estimating $\norm{\chi^2A}_{L^2(0,T; H^5(\Omega))^n}$. Recalling from \eqref{cond1} that $\chi^2A$ is supported in $(\delta,T-\delta)\times\overline{\Omega}$, we see that
$$
\norm{\chi^2A}_{L^2(0,T; H^5(\Omega))^n}^2 =  \int_{\R^{1+n}} \langle \xi \rangle^{10} |\widehat{\chi^2A}(\tau,\xi)|^2d\xi d\tau,$$
where, as usual, $\langle \xi \rangle$ denotes $\left( 1+ | \xi |^2 \right)^{\frac{1}{2}}$.
Next, for $R \in (1,+\infty)$ fixed, we put $B_R :=\{(\tau,\xi)\in\R^{1+n};\ |(\tau,\xi)|<R\}$, use that $\chi^2A\in H^1(\R; H^5(\R^n))^n\cap L^2(\R; H^6(\R^n))^n$, and obtain
\beas
\int_{\R^{1+n}\setminus B_R} \langle \xi \rangle^{10} |\widehat{\chi^2A}(\tau,\xi)|^2d\xi d\tau
& \leq & R^{-2} \int_{\R^{1+n}\setminus B_R} | (\xi,\tau) |^{2} \langle \xi \rangle^{10} |\widehat{\chi^2A}(\tau,\xi)|^2d\xi d\tau \nonumber \\
&\leq &R^{-2} \left( \norm{\chi^2A}_{H^1(\R; H^5(\R^n))^n}^2+\norm{\chi^2A}_{ L^2(\R; H^6(\R^n))^n}^2 \right).
\eeas
Therefore, we have
\bel{p2ba}
\int_{\R^{1+n}\setminus B_R} \langle \xi \rangle^{10} |\widehat{\chi^2A}(\tau,\xi)|^2d\xi d\tau \leq  CR^{-2}\delta^{-2}\norm{A}_{H^6(Q)^n}^2 \leq CR^{-2}\delta^{-2}, 
\ee
where $C$ is a generic positive constant depending only on $\Omega$, $T$ and $M$, which may change from line to line.
Further, setting $E_R:=\{(\tau,\xi)\in\R^{1+n};\ |\xi|\leq R^{-\frac{3}{n}}\}$, we get
\beas
\int_{B_R\cap E_R} \langle \xi \rangle^{10} |\widehat{\chi^2A}(\tau,\xi)|^2d\xi d\tau
&\leq & \int_{-R}^R \int_{| \xi | \leq R^{-\frac{3}{n}}} \langle \xi \rangle^{10} |\widehat{\chi^2A}(\tau,\xi)|^2d\xi d\tau\\
& \leq & (1+R^{-\frac{3}{n}})^{10} \norm{\widehat{\chi^2A}}_{L^\infty(\R^{1+n})^n}^2 \left( \int_{-R}^R \int_{| \xi | \leq R^{-\frac{3}{n}}} d\xi d\tau \right)
\eeas
and hence
$$
\int_{B_R \cap E_R} \langle \xi \rangle^{10} |\widehat{\chi^2A}(\tau,\xi)|^2d\xi d\tau \leq (2 \pi)^{-(n+1)} 2^{11} T^2 | \Omega |^2 \norm{A}_{L^\infty(Q)^n}^2 R^{-2}.
$$
This and \eqref{p2ba} yield
\bel{p2b} 
\int_{\left( \R^{1+n}\setminus B_R \right) \cup \left( B_R\cap E_R \right)} \langle \xi \rangle^{10} |\widehat{\chi^2A}(\tau,\xi)|^2 d\xi d\tau\leq CR^{-2}\delta^{-2}.
\ee

On the other hand, putting $\epsilon:=\norm{\Lambda_{A_1,q_1}-\Lambda_{A_2,q_2}}$, we derive from \eqref{l3a} for all $(\tau,\xi)\in B_R\setminus E_R$, that 
\beas
|\widehat{\chi^2A}(\tau,\xi)|^2 & \leq & C(  R^{{6\over n}+12}\delta^{-16}\sigma^{10}\epsilon^2+R^{{6\over n}+16}\delta^{-12}\sigma^{-2}) \\
& \leq & C(  R^{15}\delta^{-16}\sigma^{10} \epsilon^2+R^{19}\delta^{-12}\sigma^{-2})
\eeas
which involves
$$\int_{B_R\setminus E_R} \langle \xi \rangle^{10} | \widehat{\chi^2A}(\tau,\xi)|^2d\xi d\tau\leq C( R^{26+n}\delta^{-16}\sigma^{10}\epsilon^2+R^{30+n}\delta^{-12}\sigma^{-2}).$$
It follows from this and \eqref{p2b} that
\bel{p2c}
\norm{\chi^2A}_{L^2(0,T; H^5(\Omega))^n}^2\leq C( R^{26+n}\delta^{-16}\sigma^{10}\epsilon^2+R^{30+n}\delta^{-12}\sigma^{-2}+R^{-2}\delta^{-2}).
\ee

Further, by noticing that
$\norm{A-\chi^2A}_{L^2(0,T; H^5(\Omega))^n}\leq \norm{1-\chi^2}_{L^2(0,T)} \norm{A}_{W^{5,\infty}(Q)}$ 
and taking advantage of the fact that the $[0,1]$-valued function $1-\chi$ vanishes in $[2\delta,T-2\delta]$, we get that
$\norm{1-\chi^2}_{L^2(0,T)}^2=\int_0^{2\delta}(1-\chi^2(t))^2dt+\int_{T-2\delta}^{T}(1-\chi^2(t))^2dt \leq 4  \delta$. This entails
$\norm{A-\chi^2A}_{L^2(0,T; H^5(\Omega))^n}^2\leq 4 \norm{A}_{W^{5,\infty}(Q)}^2 \delta$ and consequently 
\bel{p2d}
\norm{A}_{L^2(0,T; H^5(\Omega))^n}^2\leq C \left(  \epsilon^2 R^{26+n}\delta^{-16}\sigma^{10}+R^{30+n}\delta^{-12}\sigma^{-2}+R^{-2}\delta^{-2}+\delta \right),
\ee
by invoking \eqref{p2c}. Now, the strategy is to choose $\delta$ as a power of $R$ so that $R^{-2}\delta^{-2}=\delta$, i.e. $\delta=R^{-\frac{2}{3}}$, and to do the same with $\sigma$, that is to take $\sigma=R^{\frac{116+3n}{6}}$, in such a way that the three last terms in the right hand side of \eqref{p2d} are equal to $R^{-\frac{2}{3}}$. Evidently, as we have $\delta \in \left( 0, \frac{T}{4} \right)$, by assumption, this requires that $R$ be fixed  in $\left( \left( \frac{T}{4} \right)^{-\frac{3}{2}}, \infty \right)$. Summing up, we infer from \eqref{p2d} that
\bel{p2e}
\norm{A}_{L^2(0,T; H^5(\Omega))^n}^2 \leq C \left( R^{230+6n} \epsilon^2 +R^{-\frac{2}{3}} \right).
\ee
Therefore, we get \eqref{t1a} with $r:=\frac{1}{346+9n}$ for all $\epsilon \in ( 0, \epsilon_r)$, where $\epsilon_r:=\left( \frac{T}{4} \right)^{\frac{1}{2r}}$, upon choosing $R=\epsilon^{-3 r}$ in \eqref{p2e}, whereas
$\norm{A}_{L^2(0,T; H^5(\Omega))^n}\leq \frac{2M}{\epsilon_r^r} \epsilon^r$ for all $\epsilon \in [\epsilon_r,+\infty)$. This achieves the proof of \eqref{t1a}.

\section{Proof of the stability estimate \eqref{t1b}}
\setcounter{equation}{0}
\label{sec-seel}

Here we use the definitions and notations introduced in Sections \ref{sec-GO} and \ref{sec-sema}, unless for the function $\beta$, which is no longer given by \eqref{GO8} but is rather defined by 
$$ \beta(t,x):=e^{-i(t \tau+ x\cdot\xi)},\ (t,x) \in (0,T) \times \R^n. $$
As this definition formally coincides with \eqref{GO8} in the particular case where $A$ is uniformly zero and $y=i\xi$, it is apparent that the estimates derived in Sections \ref{sec-GO} and \ref{sec-sema} remain valid with this specific choice of $\beta$.

Thus, in light of \eqref{GO101}--\eqref{GO112} and \eqref{l2c}-\eqref{l2d}, it holds true that
\bea
& & \abs{\int_{\R^{1+n}}V u_{1,1}\overline{u_{2,1}} (t,x)dxdt} \nonumber \\
& \leq & C \left( \norm{A}_{L^\infty(Q)^n} \langle \tau,\xi \rangle^8 \delta^{-6}\sigma +\norm{\Lambda_{A_2,q_2}-\Lambda_{A_2,q_2}}\langle \tau,\xi \rangle^6 \delta^{-8} \sigma^6+ \langle \tau,\xi \rangle^6 \delta^{-4} \sigma^{-1} \right). \label{t1d}
\eea
Next, from the very definition of $V$, we have
\beas
\int_{\R^{1+n}} Vu_{1,1}\overline{u_{2,1}}(t,x)dxdt
& =& \int_{\R^{1+n}} q u_{1,1}\overline{u_{2,1}}(t,x)dxdt - i \int_{\R^{1+n}}
A \cdot \nabla ( u_{1,1} \overline{u_{2,1}})(t,x)dxdt \\
& & -\int_{\R^{1+n}} A\cdot (A_1+A_2)u_{1,1}\overline{u_{2,1}}(t,x)dxdt,
\eeas
by applying the Stokes formula. This, \eqref{GO101}--\eqref{GO112} and \eqref{t1d} yield
\bea 
& & \abs{\int_{\R^{1+n}}qu_{1,1}\overline{u_{2,1}}(t,x)dxdt} \\
& \leq & C \left( \norm{A}_{L^\infty(Q)^n} \langle \tau,\xi \rangle^8 \delta^{-6}\sigma +\norm{\Lambda_{A_2,q_2}-\Lambda_{A_2,q_2}}\langle \tau,\xi \rangle^6 \delta^{-8} \sigma^6 + \langle \tau,\xi \rangle^6 \delta^{-4} \sigma^{-1} \right). \label{t1e}
\eea
On the other hand, we have
$$\int_{\R^{1+n}}q u_{11}\overline{u_{21}}(t,x) dxdt=\int_{\R^{1+n}} \chi^2(t) q(t,x) e^{-i(t\tau+x\cdot\xi)}e^{i\int_0^{+\infty} A(t,x+s\omega)\cdot\omega ds}dxdt, $$
whence
\beas
& & \abs{\int_{\R^{1+n}} \chi^2(t) q(t,x) e^{-i(t\tau+x\cdot\xi)}dxdt} \nonumber \\
&\leq & \abs{\int_{\R^{1+n}}q u_{11}\overline{u_{21}}(t,x)dxdt} +\abs{\int_{\R^{1+n}} q(t,x) e^{-i(t\tau+x\cdot\xi)}\left(e^{i\int_0^{+\infty} A(t,x+s\omega)\cdot\omega ds}-1\right)dxdt}.
\eeas
Thus, applying the mean value theorem, we get that 
$$ \abs{\int_{\R^{1+n}} \chi^2(t) q(t,x) e^{-i(t\tau+x\cdot\xi)}dxdt} \leq
\abs{\int_{\R^{1+n}}qu_{11}\overline{u_{21}}dxdt}+ C\norm{A}_{L^\infty(Q)^n}.
$$
Plugging \eqref{t1e} into the above estimate, we find that
\bel{t1f}
\abs{\widehat{\chi^2 q}(\tau,\xi)} \leq C \left( \norm{A}_{L^\infty(Q)^n}  \langle \tau,\xi \rangle^8 \delta^{-6}\sigma+\norm{\Lambda_{A_2,q_2}-\Lambda_{A_2,q_2}} \langle \tau,\xi \rangle^6  \delta^{-8} \sigma^6 + 
\langle \tau,\xi \rangle^6 \delta^{-4} \sigma^{-1} \right).
\ee
The next step of the proof is to upper bound $\norm{A}_{L^\infty(Q)^n}$ in terms of $\norm{\Lambda_{A_1,q_1}-\Lambda_{A_2,q_2}}$. To this end, we pick $p \in (n+1,+\infty)$, apply the Sobolev embedding theorem (see e.g. \cite[Corollary IX.14]{Br}) and find that $\norm{A}_{L^\infty(Q)^3}\leq C \norm{A}_{W^{1,p}(Q)^n}$. Interpolating, we obtain 
$$ \norm{A}_{L^\infty(Q)^n}\leq C \norm{A}_{W^{2,p}(Q)^n}^{1 \slash 2} \norm{A}_{L^p(Q)^n}^{1 \slash 2}\leq C \norm{A}_{L^p(Q)^n}^{1 \slash 2} \leq C\norm{A}_{L^2(Q)^n}^{1 \slash p}
$$
and hence
$$ \norm{A}_{L^\infty(Q)^n} \leq C\norm{\Lambda_{A_1,q_1}-\Lambda_{A_2,q_2}}^{r \slash p}, $$
with the help of \eqref{t1a}. Inserting the above estimate into \eqref{t1f} then yields
$$
\abs{\widehat{\chi^2 q}(\tau,\xi)} 
\leq C \left( \norm{\Lambda_{A_1,q_1}-\Lambda_{A_2,q_2}}^{r \slash p}  \langle \tau,\xi \rangle^8 \delta^{-6}\sigma +\norm{\Lambda_{A_2,q_2}-\Lambda_{A_2,q_2}} \langle \tau,\xi \rangle^6 \delta^{-8} + \langle \tau,\xi \rangle^6 \delta^{-4} \sigma^{-1} \right)
$$
and \eqref{t1b} follows from this by arguing in the same way as in the derivation of \eqref{t1a} from \eqref{p2d}.



\end{document}